\documentclass[11pt]{amsart}

\def\A{\mathcal A}

\def\C{\mathcal C}

\def\P{\mathcal P}

\def\P{\mathcal P}

\def\Lu{{\mathcal L}(\uu)}
\def\uu{\mathbf u}

\newtheorem{thm}{Theorem}

\newtheorem{coro}[thm]{Corollary}

\newtheorem{lem}[thm]{Lemma}
\newtheorem{lemma}[thm]{Lemma}

\newtheorem{pozn}[thm]{Remark}

\newtheorem{prop}[thm]{Proposition}
\newtheorem{defi}[thm]{Definition}

\newtheorem*{conjBrRe}{Brlek-Reutenauer Conjecture}

\theoremstyle{remark}

\usepackage[utf8]{inputenc}
\usepackage[T1]{fontenc}

\begin{document}
\title{Proof of Brlek-Reutenauer conjecture}

\author{L{\!'}. Balková$^1$ and E. Pelantov\'a$^1$}
\address{$^1$Department of Mathematics, FNSPE, Czech Technical University in Prague, Trojanova 13, 120~00 Praha 2, Czech Republic}

\author{\v S. Starosta$^2$}
\address{$^2$Department of Applied Mathematics, FIT, Czech Technical University in Prague, Thákurova 9, 160~00 Praha 6, Czech Republic}

\maketitle
\begin{abstract}
Brlek and Reutenauer conjectured that any infinite word ${\mathbf
u}$ with language closed under reversal satisfies the equality
$2D({\mathbf u}) = \sum_{n=0}^{+\infty}T_{\mathbf u}(n)$ in which
$D({\mathbf u})$ denotes the defect of $\mathbf u$ and $T_{\mathbf
u}(n)$ denotes $\mathcal{C}_{\mathbf u}(n+1)-\mathcal{C}_{\mathbf
u}(n)  +2 - \mathcal{P}_{\mathbf u}(n+1) - \mathcal{P}_{\mathbf
u}(n)$, where $\mathcal{C}_{\mathbf u}$  and $
\mathcal{P}_{\mathbf u}$ are the factor and palindromic complexity
of ${\mathbf u}$, respectively. This conjecture was verified
for periodic words by Brlek and Reutenauer themselves. Using
their results for periodic words, we have recently proved the conjecture for
uniformly recurrent words. In the present article we prove the
conjecture in its general version by a new method without
exploiting the result for periodic words.

\end{abstract}
\section{Introduction}
Brlek and Reutenauer conjectured in~\cite{BrRe-conjecture} a nice
equality which combines together  the  factor complexity
$\mathcal{C}_{\mathbf u}$, the palindromic complexity  $
\mathcal{P}_{\mathbf u}$, and the palindromic defect  $D(\uu)$ of
an infinite word $\uu$.  It sounds as follows.

\begin{conjBrRe} 
If ${\mathbf u}$ is an infinite word with language closed under
reversal, then
$$2D(\uu)=\sum_{n=0}^{+\infty}T_{\mathbf u}(n)\,,$$
where  $T_{\mathbf u}(n)=\mathcal{C}_{\mathbf
u}(n+1)-\mathcal{C}_{\mathbf u}(n) +2 - \mathcal{P}_{\mathbf
u}(n+1) - \mathcal{P}_{\mathbf u}(n)$\,.
\end{conjBrRe}

Brlek and Reutenauer proved ibidem that their conjecture holds for
periodic infinite words. It is known from~\cite{BuLuGlZa} that the
Brlek-Reutenauer conjecture holds for words with zero defect.
In~\cite{BaPeSta4}, we proved the conjecture for uniformly
recurrent words.  In our proof,   we constructed   for any
uniformly recurrent word   $\uu$ whose language is closed under
reversal  a periodic word ${\mathbf v}$ with language closed under
reversal such that $D({\mathbf u}) = D({\mathbf v})$ and
$T_{\mathbf u}(n) = T_{\mathbf v}(n)$  for any $n$. Then we used
validity of the conjecture for periodic words.

In this paper, we will prove that the Brlek-Reutenauer conjecture
holds in full generality without exploiting the result for periodic words.
Since both sides of the equality in the Brlek-Reutenauer conjecture are non-negative, validity of the conjecture will be shown if we prove the following two theorems.

\begin{thm}\label{konecne}
If ${\mathbf u}$ is an infinite word with language closed under
 reversal such that  both  $D(\uu)$ and  $\sum_{n=0}^{+\infty}T_{\mathbf u}(n)$ are finite,
then
\begin{equation}
\label{BRconj} 2D({\mathbf u}) = \sum_{n=0}^{+\infty}T_{\mathbf
u}(n)\,.
\end{equation}
\end{thm}

\begin{thm}\label{finiteness}
If ${\mathbf u}$ is an infinite word with language closed under
reversal, then $$D({\mathbf u})<+\infty\qquad  \hbox{ if and only
if}\qquad \sum_{n=0}^{+\infty}T_{\mathbf u}(n)<+\infty\,.$$
\end{thm}
In the paper~\cite{BaPeSta4} which is  devoted mainly to the
uniformly recurrent words, we already stated in the section
Open problems one part of Theorem \ref{finiteness}, namely that
$D({\mathbf u})<+\infty$ implies $\sum_{n=0}^{+\infty}T_{\mathbf
u}(n)<+\infty\,.$ As pointed out in  \cite{Basic12}, there is a
gap in our proof, and its corrected version   can be found
in~\cite{BaPeSta4_Corrigendum}. In order to make the present paper
self-sustained so that the reader understand and check all steps
of the proof without having all previous papers at hand, we recall
necessary notations and statements together with the proofs of the
essential ones.

\section{Preliminaries}
By $\mathcal{A}$ we denote a~finite set of symbols called
{\em letters}; the set $\mathcal{A}$ is therefore called an {\em
alphabet}. A finite string $w=w_0w_1\ldots w_{n-1}$ of letters from
$\mathcal{A}$ is said to be a~{\em finite word}, its length is
denoted by $|w| = n$. Finite words over $\mathcal{A}$ together
with the operation of concatenation and the empty word $\epsilon$
as the neutral element form a~free monoid $\mathcal{A}^*$. The map
$$w=w_0w_1\ldots w_{n-1} \quad \mapsto \quad \overline{w} =
w_{n-1}w_{n-2}\ldots w_{0}$$ is a bijection on $\mathcal{A}^*$,
the word $\overline{w}$ is called the {\em reversal} or the {\em
mirror image} of $w$. A~word $w$ which coincides with its mirror
image is a~{\em palindrome}.

Under an {\em infinite word} we understand an infinite string
${\mathbf u}=u_0u_1u_2\ldots $ of letters from $\mathcal{A}$.
A~finite word $w$ is a~{\em factor} of a~word $v$ (finite or
infinite) if there exist words $p$ and $s$ such that $v= pws$. If
$p = \epsilon$, then $w$ is said to be a~{\em prefix} of $v$, if
$s = \epsilon$, then $w$ is a~{\em suffix} of~$v$.

The {\em language} $\mathcal{L}(v)$ of a~finite or an infinite word
$v$ is the set of all its factors. Factors of $v$ of length $n$
form the set denoted by $\mathcal{L}_n(v)$. We say that the language of an infinite word ${\mathbf u}$ is {\em closed under reversal} if
$\mathcal{L}({\mathbf u})$ contains with every factor $w$ also its
reversal $\overline{w}$.

For any factor $w\in \mathcal{L}({\mathbf u})$, there exists an
index $i$ such that $w$ is a prefix of the infinite word
$u_iu_{i+1}u_{i+2} \ldots$. Such an index is called an {\em
occurrence} of $w$ in ${\mathbf u}$. If each factor of $\mathbf u$
has infinitely many occurrences in ${\mathbf u}$, the infinite
word $\mathbf u$ is said to be {\em recurrent}. It is easy to see
that if the language of ${\mathbf u}$ is closed under reversal,
then ${\mathbf u}$ is recurrent (a~proof can be found
in~\cite{GlJuWiZa}). For a~recurrent infinite word ${\mathbf u}$,
we may define the notion of a~{\em complete return word} of any $w
\in\mathcal{L}({\mathbf u})$. It is a~factor $v\in
\mathcal{L}({\mathbf u})$ such that $w$ is a prefix and a suffix
of $v$ and $w$ occurs in $v$ exactly twice.

If any factor $w \in \mathcal{L}({\mathbf u})$ has
only finitely many complete return words, then the infinite word ${\mathbf
u}$ is called {\em uniformly recurrent}.

The {\em factor complexity} of an infinite word ${\mathbf u}$ is
the~map $\mathcal{C}_{\bf u}: \mathbb{N} \mapsto \mathbb{N}$ defined
by the prescription $\mathcal{C}_{\bf u}(n):=\#
\mathcal{L}_n({\mathbf u})$.
To determine the first difference of
the factor complexity, one has to count the possible {extensions}
of factors of length $n$. A~{\em right extension} of $w \in
\mathcal{L}({\mathbf u})$ is a letter $a\in \mathcal{A}$ such
that $w a\in \mathcal{L}({\mathbf u})$. Of course, any factor of
${\mathbf u}$ has at least one right extension. A~factor $w$ is
called {\em right special} if $w$ has at least two right
extensions. Similarly, one can define a~{\em left extension} and
a~{\em left special} factor. We will deal mainly with recurrent
infinite words ${\mathbf u}$. In such a~case, any factor of $\mathbf
u$ has at least one left extension.

In~\cite{DrJuPi} it is shown that any finite word $w$ contains at most $|w|+1$ distinct palindromes (including the empty word).
The {{\em defect} $D(w)$ of a~finite word $w$ is the difference
between the utmost number of palindromes $|w|+1$ and the actual
number of palindromes contained in $w$.

In accordance with the terminology introduced in \cite{DrJuPi},
the factor with a~unique occurrence in another factor is called
{\em unioccurrent}.

The following corollary gives an insight into the birth of defects.
\begin{coro}[\cite{DrJuPi}]\label{prodlouzeni} The defect $D(w)$ of a~finite word $w$ is equal to the
number of prefixes $w'$ of $w$ for which the longest palindromic
suffix of $w'$ is not unioccurrent in $w'$. In other words, if $b$
is a~letter and $w$ a~finite word, then $D(wb)=D(w)+\delta$, where
$\delta = 0$ if the longest palindromic suffix of $wb$ occurs
exactly once in $wb$ and $\delta = 1$ otherwise.
\end{coro}

Corollary~\ref{prodlouzeni} implies that $D(v) \geq D(w)$ whenever $w$ is
a~factor of $v$. It enables to give a~reasonable definition of the
defect of an infinite word (see~\cite{BrHaNiRe}).

\begin{defi}\label{defekt}  The defect of  an  infinite word  $\mathbf{u}$ is the  number (finite or infinite)
$$D({\mathbf{u}}) = \sup \{ D(w) \colon w \ \text{is a~prefix of $\mathbf u$}\}\,.$$
\end{defi}
Let us point out two facts.
\begin{enumerate}
\item If we consider all factors of a~finite or an infinite word
$\mathbf u$, we obtain the same defect, i.e.,
$$D({\mathbf{u}}) = \sup \{ D(w) \colon w \in \mathcal{L}(\mathbf{u})\} \,.$$
\item Any infinite word with finite defect contains infinitely
many palindromes.
\end{enumerate}

Using Corollary~\ref{prodlouzeni} and Definition~\ref{defekt}, we obtain immediately the following corollary.
\begin{coro}\label{charakterizace_defekt}
Let $\mathbf u$ be an infinite word with language closed under
reversal. The
following statements are equivalent.
\begin{enumerate}
 \item The defect of ${ \mathbf u}$ is finite.
\item There exists an integer $H$ such that the longest
palindromic suffix of any prefix $w$ of length $|w| \geq H$
occurs in $w$ exactly once.
\end{enumerate}
\end{coro}

For \emph{the longest palindromic suffix} of a~word $w$ we will sometimes use the notation $lps(w)$.

\medskip

The number of  palindromes of a~fixed length occurring in an
infinite word is measured by the so called {\it palindromic
complexity} ${\mathcal{P}_{\bf u}}$, the~map which assigns to any
non-negative integer $n$ the number
$$ {\mathcal{P}_{\bf u}}(n) := \#\{ w \in {\mathcal{L}_n}(u) \colon w \ \
\hbox{is a palindrome}\}\,.$$

Denote by
$$
T_{\mathbf u}(n) = \mathcal{C}_{\mathbf
u}(n+1)-\mathcal{C}_{\mathbf u}(n)  +2 - \mathcal{P}_{\mathbf
u}(n+1) - \mathcal{P}_{\mathbf u}(n).
$$
The following proposition is proven in \cite{BaMaPe} for uniformly
recurrent words, however, as also noted in \cite{BrRe-conjecture}, the uniform recurrence is not needed in
the proof and it holds for any infinite word with language
closed under reversal.

\begin{prop}[\cite{BaMaPe}]\label{Balazi}
 If $\mathbf u$ is an infinite word
with language closed under reversal, then
\begin{equation}\label{BalaziNerovnost}
T_{\mathbf u}(n)\geq 0 \quad \text{for all $n\in \mathbb{N}$.}
\end{equation}

\end{prop}

Let ${\mathbf u}$  be an infinite word with language closed under
reversal. Using the proof of Proposition~\ref{Balazi}, those $n \in
\mathbb N$ for which $T_{\mathbf u}(n)=0$ can be characterized in
the graph language.
Before doing that we need to introduce some more notions.

An {\em $n$-simple path} $e$ is a~factor of ${\mathbf u}$ of
length at least $n + 1$ such that the only special (right or left)
factors of length $n$ occurring in $e$ are its prefix and suffix
of length $n$. If $w$ is the prefix of $e$ of length $n$ and $v$
is the suffix of $e$ of length $n$, we say that the $n$-simple
path $e$ starts in $w$ and ends in $v$. We will denote by
$G_n({\mathbf u})$ an undirected graph whose set of vertices is
formed by unordered pairs $(w,\overline{w})$ such that $w \in
\mathcal{L}_n({\mathbf u})$ is right or left special. We connect
two vertices $(w,\overline{w})$ and $(v,\overline{v})$ by an
unordered pair $(e,\overline{e})$ if $e$ or $\overline{e}$ is an
$n$-simple path starting in $w$ or $\overline{w}$ and ending in
$v$ or $\overline{v}$. Note that the graph $G_n({\mathbf u})$ may
have multiple edges and loops.

\begin{lem}\label{graph}
If ${\mathbf u}$  is an infinite  word with language
closed under reversal and $n\in \mathbb{N}$, then $T_{\mathbf u}(n)
= 0$ if and only if both of the following conditions are
met.
\begin{enumerate}\item The graph obtained from $G_n({\mathbf u})$ by
removing loops is a tree. \item Any $n$-simple path forming a loop
in the graph $G_n({\mathbf u})$  is a palindrome.
\end{enumerate}
\end{lem}
\begin{proof} It is a~direct consequence of the proof of Theorem 1.2 in~\cite{BaMaPe}
(recalled in this paper as Proposition~\ref{Balazi}).
\end{proof}

\section{Proof of Theorem \ref{konecne}}

The aim of this section is to prove Theorem~\ref{konecne},
i.e., to prove the Brlek-Reutenauer conjecture under the
additional assumption that the defect $D(\uu)$ of an infinite
word $\uu$ and the sum $\sum_{n=0}^{\infty}T_{\uu}(n)$ are
finite. As observed in~\cite{BrRe-conjecture}, it is easy to prove
the ``finite analogy'' of the conjecture, which deals only with
finite words. We will also make use of  this result.
\begin{thm}[\cite{BrRe-conjecture}]\label{BRconj_finite}
For every finite word $w$ we have
$$2D(w)=\sum_{n=0}^{|w|}T_w(n),$$ where $T_w(n)=
\mathcal{C}_{w}(n+1)-\mathcal{C}_{w}(n)  +2 - \mathcal{P}_{w}(n+1)
- \mathcal{P}_{w}(n)$ and the index $w$ means that we consider
only factors of $w$.
\end{thm}
It may seem that the Brlek-Reutenauer conjecture for an infinite word
$\uu$ can be obtained from Theorem \ref{BRconj_finite} by a
``limit transition''.
However, this transition would be far from being kosher.  The following
lemmas enable us to avoid the incorrectness.

\medskip

\begin{lemma}\label{lubka1}
Let ${\mathbf u}$ be an infinite word with language closed under reversal and finite defect.
If $q$ is its
prefix satisfying $D({\mathbf u})= D(q)$, then for $H=|q|
+1$ one has $$\C_{\uu}(H)-\P_{\uu}(H)=2\#\{x \in {\mathcal{L}}(\uu)
\colon x \ \text{is a palindrome shorter than $H$ which is not contained
in $q$}\}.$$

\end{lemma}

\begin{proof} Let us define a mapping  $f: S\to T$, where
$$S = \{ x \in \mathcal{L}({\mathbf u}) \colon
x \notin  \mathcal{L}(q), \,|x| < H ,\, x = \overline{x} \}$$ and
$$T= \bigl\{ \{ w, \overline{w}\} \colon  w \in {\mathcal{L}}_H(\uu),\, w
\neq \overline{w}\bigr\}\,.$$ Let $x$ be a palindrome from $S$ and
$i$ be the first occurrence of $x$ in ${\mathbf u}$. Put $w =
u_{i+|x|-H} \cdots u_{i+{|x|-1}}$. It means that $w$ is a factor
of ${\mathbf u}$ of length $H$ and $x$ is a suffix of $w$. Since
$H
> |x|$,  the factor $w$ is not a palindrome - otherwise it
contradicts the fact that $i$ is the first occurrence of the palindrome
$x$. We put $f(x) = \{w, \overline{w}\}$.\\

To show that $f$  is surjective, we consider $w \in
{\mathcal{L}}_H(\uu)$ such that $w \neq \overline{w}$. Let $p$ be
the prefix of $\mathbf{u}$ which ends in the first occurrence of
$w$ or $\overline{w}$ in $\mathbf{u}$.  Since $|p|\geq H = |w|
> |q|$, we  have according to Corollary  \ref{prodlouzeni} that $D(q) = D(p)$ and consequently, $lps(p)$ is
unioccurrent in $p$, which implies that $lps(p)$ is not a factor
of $q$. Moreover, $lps(p)$ is shorter than $H$ - otherwise it
contradicts the choice of the prefix $p$. We found $x =
lps(p) \in S$ such that  $f(x) = \{w,\overline{w}\}$, i.e.,  $f$
is surjective.

To show that $f$ is injective, we consider two palindromes $ y,z
\in S$ and we denote $f(y) = \{w_y, \overline{w_y}\}$ and $f(z) =
\{w_z, \overline{w_z}\}$. From the definition of $w_x$ we know
that the palindrome $x$ occurs as a factor of $w_x$ exactly once,
namely  as its suffix. It means that $x$ equals $lps(w_x)$.
 Let us suppose that $f(y) = f(z)$. 
 We have to discuss two cases.

\begin{enumerate}
\item Case  $w_y=w_z$.  It gives $ lps(w_y) = lps(w_z) $ and thus
 $y=z$.
\item Case $w_y=\overline{w_z}$. It implies that $y$ is a prefix of
$w_z$ and $z$ is a prefix  of $w_y$.   The fact that $y$ is a
prefix of $w_z$ forces  the first occurrence of $w_y$ to be
strictly smaller  than the first occurrence of $w_z$.
Simultaneously, since $z$ is a prefix of $w_y$, the first
occurrence of $w_z$ is strictly smaller  than the first occurrence
of $w_y$ - a contradiction.
\end{enumerate}

Consequently, the assumption $f(y) = f(z)$
implies $z=y$ and  the mapping $f$ is injective as well.

Existence of the  bijection $f$ between the finite sets $T$  and
$S$ means $\# T = \#S$.
Since from the definition of $T$ it follows that  $ \C_{\uu}(H)-\P_{\uu}(H) = 2\#T$,
the equality  stated in the lemma is proven.

\end{proof}
\begin{pozn}
As it was pointed out by Bojan Ba\v si\'{c}, Lemma~\ref{lubka1} may be stated in a~more general form for $H>|q|$, then the equality changes to 
$$\C_{\uu}(H)-\P_{\uu}(H)=2\# \{ x \in \mathcal{L}({\mathbf u}) \colon
x \notin  \mathcal{L}(q), \,|x| < H ,\, x = \overline{x} \}-2(H-|q|-1).$$
Thanks to him, we added the assumption $H=|q|+1$ in Lemma~\ref{lubka1} necessary for the validity of the statement.
\end{pozn}

\begin{lemma}\label{lubka2}  Let ${\mathbf u}$ be an infinite word with language closed under reversal and finite defect.
If $q$ is its
prefix satisfying $D({\mathbf u})= D(q)$, then for any prefix $p$
of  ${\mathbf u}$  such that $|p|> |q|$ the number
$$\#\{x \in {\mathcal{L}}(p)\colon  x \
\text{is a palindrome of length at most $|q|$ which is not contained
in $q$}\} \ +  \sum_{n=|q|+1}^{|p|}\!\!\P_p(n)$$ equals $|p|-|q|$.
\end{lemma}

\begin{proof}
At first we will show the equality
 \begin{equation}\label{pomoc}|p| - |q| = \# \{x
\in {\mathcal{L}}(p) \setminus {\mathcal{L}}(q)\colon
x=\overline{x}\}\,.\end{equation}  Let us denote by $u^{(i)}$ the
prefix of ${\uu}$ of length $i$. For any palindrome $x \in
{\mathcal{L}}(p) \setminus {\mathcal{L}}(q)$ we find the minimal index $i$
such that  $x$ occurs in $u^{(i)}$. Since $x \in {\mathcal{L}}(p)
\setminus {\mathcal{L}}(q)$,   we have $|q| < i \leq |p|$. Thus we map any
element of $ \{x \in {\mathcal{L}}(p) \setminus {\mathcal{L}}(q)\colon
x=\overline{x}\}$  to an index $i \in \{ |q|+1, |q|+2, \ldots,
|p|\}$.

 Let us look at the details of this
mapping.  The minimality of $i$ guarantees that $x$ is
unioccurrent in $u^{(i)}$. Palindromicity of $x$ gives that $x
=lps(u^{(i)})$. It implies that no two different  palindromes are
mapped to the same index $i$, i.e.,  the mapping is injective.

Since $D(q) = D(\uu)$, according to Corollary \ref{prodlouzeni},
$lps(u^{(i)})$  is unioccurrent in $u^{(i)}$ and thus
$lps(u^{(i)})\notin  {\mathcal{L}}(q)$.      Thus any index $i$
such that $|q| <i  \leq |p|$ has its preimage $x=lps(u^{(i)})$.
Therefore the mapping is a bijection and its domain and range have
the same cardinality as stated in \eqref{pomoc}.

To finish the proof, we split elements of  $\{x \in {\mathcal{L}}(p) \setminus
{\mathcal{L}}(q)\,:\, x=\overline{x}\}$ into two disjoint parts:
 elements of length smaller than or equal to $ |q|$ and elements of length
greater than $|q|$. Since
$$\#\{x \in {\mathcal{L}}(p) \setminus {\mathcal{L}}(q)\,:\,
x=\overline{x}, \ |x| >|q|\} = \#\{x \in {\mathcal{L}}(p) \,:\,
x=\overline{x}, \ |x| >|q|\}  = \sum_{n=|q|+1}^{|p|}\!\!\P_p(n)\,,
$$
 the statement of Lemma~\ref{lubka2} is proven.

\end{proof}

Now we can complete the proof of Theorem \ref{konecne}.
\begin{proof}[Proof of Theorem~\ref{konecne}]
Finiteness of defect means that there exists a  constant $L\in
\mathbb{N}$ such that $D(\uu) = D(q)$ for any prefix $q$ of $\uu$
which is longer than or of length equal to $L$. On the other hand, finiteness of the sum
$\sum_{n=0}^{+\infty}T_{\mathbf u}(n)$ together with the fact
$0\leq T_{\mathbf u}(n) \in \mathbb{Z}$ for any $n \in \mathbb{N}$
implies that there exists a  constant $M\in \mathbb{N}$ such that
$ T_{\mathbf u}(n) = 0$ for any  $n > M$. Let us fix an integer
$H> \max\{L,M\}$ and denote by $q$ the prefix of $\uu$ of length
$|q|=H-1$. Consequently,
$$  T_{\mathbf u}(n) = 0 \ \ \hbox{ for any \ }  n\geq H\quad
\hbox{and } \quad D(\uu) = D(q)\,.
$$
 In order to show the equality~\eqref{BRconj}, it thus
remains to show $2D(q)=\sum_{n=0}^{H-1}T_{\uu}(n)$.

Let us consider a~prefix $p$ of $\uu$ containing all factors of
length $H$. In this case $p$ is longer than $q$, thus it holds by
Corollary~\ref{prodlouzeni} that $D(q)=D(p)$. Using
Theorem~\ref{BRconj_finite}, we have
$$2D(p)=\sum_{n=0}^{|p|}T_p(n)=\sum_{n=0}^{H-1}T_{p}(n)+\sum_{n=H}^{|p|}T_p(n)=\sum_{n=0}^{H-1}T_{\uu}(n)+\sum_{n=H}^{|p|}T_p(n),$$
where the last equality is due to the fact that $p$ contains all
factors of length $H$. It remains to prove that
$\sum_{n=H}^{|p|}T_p(n)=0$. Let us rewrite the sum by definition.
\begin{equation}\label{lubka3}\begin{array}{rcl}\sum_{n=H}^{|p|}T_p(n)&=&\sum_{n=H}^{|p|}\left(\C_p(n+1)-\C_p(n)+2-\P_p(n+1)-\P_p(n)\right)\\
&=&-\C_p(H)+2(|p|-H+1)-2\sum_{n=H}^{|p|}\P_p(n)+\P_p(H)\\
&=&-\C_{\uu}(H)+2(|p|-H+1)-2\sum_{n=H}^{|p|}\P_p(n)+\P_{\uu}(H),
\end{array}
\end{equation}
where in the last equality we again used the fact that $p$
contains all factors of length $H$. This fact also allows us to
rewrite the set $\{x \in {\mathcal{L}}(p)\, : \,  x \notin
{\mathcal{L}}(q)\,, \ x=\overline{x}\,,\ |x|\leq |q| \} $ from
Lemma \ref{lubka2} as $\{x \in {\mathcal{L}}(\uu)\, : \,  x \notin
{\mathcal{L}}(q)\,, \ x=\overline{x}\,, \ |x| <H\} $. Denote the
cardinality of this set by $B$.

In this notation,  Lemma \ref{lubka1} and Lemma \ref{lubka2} say
$$
\C_{\uu}(H)-\P_{\uu}(H) =2B \quad \hbox{and } \quad  B+
\sum_{n=H}^{|p|}\!\!\P_p(n)= |p| -H+1\,.$$ This implies that the
last expression in \eqref{lubka3} is zero as desired.

\end{proof}

\section{Proof of Theorem \ref{finiteness}}

If an infinite word $\uu$ is periodic with language closed under reversal, then $D(\uu)<+\infty$ and $\sum_{n=0}^{+\infty}T_{\mathbf u}(n)<+\infty$, as shown in~\cite{BrRe-conjecture}.
Consequently, we will limit our considerations in the sequel to aperiodic words.
\begin{prop}\label{charakterizace_T}
If $\uu$ is an aperiodic infinite word with language closed under reversal and $N$ is an integer,
then $T_{\mathbf u}(n)=0$ for all $n \geq N$ if and only if for any factor $w$ such that $|w| \geq N$,
 any factor longer than $w$ beginning in $w$ or $\overline{w}$ and ending in $w$ or $\overline{w}$, with no other
 occurrences of $w$ or $\overline{w}$, is a palindrome.
\end{prop}
\begin{proof}
$(\Leftarrow):$ Let us show for any $n \geq N$ that the assumptions of Lemma~\ref{graph} are satisfied.
We have to show two properties of $G_n(\uu)$ for any $n \geq N$.
\begin{enumerate}
\item Any loop in $G_n(\uu)$ is a palindrome. \\
Since any loop $e$ in $G_n(\uu)$ at a vertex $(w,\overline{w})$ is a~word longer than $w$ beginning in a~special factor $w$ or $\overline{w}$ and ending in $w$ or $\overline{w}$, with no other occurrences of $w$ or $\overline{w}$, the loop $e$ is a~palindrome by the assumption.
\item
The graph obtained from $G_n(\uu)$ by removing loops is a tree. \\
Or equivalently, we have to show that in $G_n(\uu)$ there exists a~unique path between any two different vertices $(w',\overline{w'})$ and $(w'',\overline{w''})$.
Let $p$ be a~factor of $\uu$ such that $w'$ or $\overline{w'}$ is its prefix, $w''$ or $\overline{w''}$ is its suffix and $p$ has no
other occurrences of $w', \ \overline{w'}, \ w'', \ \overline{w''}$.
Let $v$ be a~factor starting in $p$, ending in $w'$ or $\overline{w'}$ and containing no other occurrences of $w'$ or $\overline{w'}$.
By the assumption the factor $v$ is a palindrome, thus $\overline{p}$ is a~suffix
of $v$.
It is then a~direct consequence of the construction of $v$ that the next factor with the same properties as $p$, i.e., representing a~path in the
undirected graph $G_n(\uu)$ between $w'$ and $w''$, which
occurs in $\uu$ after $p$, is $\overline{p}$.
This shows that there is only one such path.
\end{enumerate}
Consequently, Lemma~\ref{graph} implies that $T_n(\mathbf{u}) = 0$ for any $n \geq N$.

\medskip

\noindent $(\Rightarrow):$ First  we prove an  auxiliary claim.

\medskip

\noindent {Claim:} {\it If $\uu$ is an aperiodic infinite word with
language closed under reversal and $N$ is an integer such that
$T_{\mathbf u}(n)=0$ for all $n \geq N$, then for any  $w$ such
that $|w| \geq N$ and
 any factor $v$ longer than $w$ beginning in $w$  and ending in $w$ or $\overline{w}$, with no other
 occurrences of $w$ or $\overline{w}$, there exists a letter $ a \in \A$ such that $v$ has prefix $wa$ and suffix
 $a\overline{w}$.}\\[1mm]

 It is clear that repeated application of the previous claim to factors $w$
 of length gradually increased by one gives the proof of
 implication $(\Rightarrow)$ of Proposition~\ref{charakterizace_T}. \\[1mm]

We split the proof of the auxiliary claim into two cases.

\begin{itemize}
\item Case 1: Assume that $w$ is a special factor.

If $v$ does not contain any other special factor of length
$n=|w|$ except for $w$ and $\overline{w}$, then $v$ is a loop in the
graph $G_n(\uu)$ and according to Lemma~\ref{graph}, the factor
$v$ is a palindrome. Necessarily,  $v$ begins in $wa$ for some
letter $a$ and ends in $a\overline{w}$.

Suppose now that $v= v_0v_1\cdots v_m$ contains a special factor
$z\neq w, \overline{w}$ of length $n$ at the position $i$, i.e.,
$z = v_iv_{i+1}\cdots v_{n+i-1}$. Without loss of generality, we
consider the smallest index $i$ with this property.  The pair
$(z, \overline{z})$ is a vertex in the graph $G_n(\uu)$ and a
prefix of $v$, say $e$,  corresponds to  an edge  in $G_n(\uu)$
starting in $(w, \overline{w})$ and ending in $(z, \overline{z})$.
Since the graph $G_n(\uu)$ is a tree,  the word $v$ which
corresponds to a walk from     $(w, \overline{w})$ to  the same
vertex $(w, \overline{w})$  has a suffix $f$   representing an
edge in $G_n(\uu)$ connecting again vertices   $(z, \overline{z})$
and $(w, \overline{w})$. It means that the suffix $f$ starts in
$z$ or $ \overline{z}$  and  ends in $w$ or $ \overline{w}$. Since
$G_n(\uu)$ has no multiple edges connecting distinct vertices, necessarily  $f = \overline{e}$,
which already gives the claim.

\item Case 2: Assume $w$ is not a special factor.

 It means that
there exists a unique letter $a$ such that $wa$ belongs to the
language of $\uu$. As  the language is closed under reversal, the
factor $\overline{w}$ has a unique left extension, namely
$a$.  If $v$ starts in $w$ and ends in $\overline{w}$,
then the claim is proven.

It remains to exclude  that
$v$ begins and ends in a~non-palindromic factor $w$.  Suppose this situation happens. In this case,
there exists a unique $q$ such that $wq$ is a right special factor and it is the shortest right special factor having the prefix $w$.
The factor $wq$ has only one occurrence of the factor $w$ - otherwise we
 can find a shorter prolongation of $w$ which is right special.  Since $w$ is a suffix of $v$, we deduce that $|wq| < |v|$.
Because $wq$ is the shortest right special factor with prefix
$w$,  the factor $vq$  belongs to the language and its prefix and
suffix  $wq$ is a special factor. According to already proven  Case 1, we
have $wq = \overline{wq} = \overline{q}\,\overline{w}$.  It means
together with the inequality  $|wq| < |v|$ that $\overline{w}$  is
contained in $v$  as well - a contradiction.
\end{itemize}
\end{proof}

The proof of the implication $(\Rightarrow)$ of Proposition~\ref{charakterizace_T} is taken from~\cite{PeSt}, where we showed a~more
general statement for an infinite word whose language is closed
under a~larger group of symmetries.

\begin{coro}\label{alternace}
Let $\uu$ be an aperiodic infinite word with language closed under reversal and let $N$ be an integer. If $T_{\mathbf u}(n)=0$ for all $n \geq N$, then the occurrences of $w$ and $\overline{w}$ in $\mathbf u$ alternate for any factor $w$ of $\mathbf u$ of length at least $N$.
\end{coro}

The following lemma builds a bridge between Corollary~\ref{charakterizace_defekt}
and Proposition~\ref{charakterizace_T}.
\begin{lemma}\label{lps_crw}
Let $\uu$ be an aperiodic infinite word with language closed under reversal. There exists $H \in \mathbb N$ such that the longest palindromic suffix of any prefix $w$ of $\uu$ of length $|w| \geq H$
occurs in $w$ exactly once if and only if there exists $N\in \mathbb N$ such that for any factor $w$ with $|w| \geq N$, any factor longer than $w$ beginning in $w$ or $\overline{w}$ and ending in $w$ or $\overline{w}$, with no other occurrences of $w$ or $\overline{w}$, is a palindrome.
\end{lemma}
\begin{proof}
$(\Rightarrow):$ We will show that $N$ may be set equal to $H$. Let us proceed by contradiction. Suppose there exists a~factor $w \in \Lu$ such that $|w| \geq H$
and there exists a~non-palindromic factor of $\uu$ longer than $w$ beginning in $w$ or $\overline{w}$ and ending in $w$ or $\overline{w}$, with no other occurrences of $w$ or $\overline{w}$.
Let us find the first non-palindromic factor of the above form in $\uu$ and let us denote it $r$.
Let $p$ be the prefix of $\uu$ ending in the first occurrence of $r$ in $\uu$, i.e., $p = tr$ for some word $t$ and $r$ is unioccurrent in $p$.
Denote by $s$ the longest palindromic suffix of $p$. By the assumption, $s$ is unioccurrent in $p$.
No matter how long the suffix $s$ is, we will obtain a contradiction.
\begin{enumerate}
\item If $|s| \leq |w|$, then we have a contradiction to the unioccurrence of $s$.
\item If $|r| > |s| > |w|$, then we can find at least $3$ occurrences of $w$ or $\overline{w}$ in $r$ which is a
contradiction to the form of $r$.
\item The equality $|r| = |s|$ contradicts the fact that we
supposed $r$ to be non-palindromic.
\item Finally, if $|r| < |s|$, then there is an occurrence of the mirror image of $r$ which is a non-palindromic factor having the same properties as $r$ which occurs before $r$ and contradicts the choice of $p$.
\end{enumerate}

\noindent $(\Leftarrow):$ Take a~prefix containing all factors of length $N$. Set $H$ equal to its length.
Let us show that any prefix $p$ of length greater than or equal to $H$ has $lps(p)$ of length greater than or equal to $N$.
Consider a~suffix of $p$ of length $N$, say $w$. Either $w$ is a~palindrome, then $lps(p)$ is of length greater than or equal to $N$. Or $w$ is not a~palindrome, then we find a~suffix of $p$ beginning in $\overline{w}$ and containing exactly two occurrences of $w$ or $\overline{w}$. Such a~suffix exists since all factors of length $N$ are contained in $p$. By assumptions, such a~suffix is a~palindrome, hence $lps(p)$ is longer than $N$.

Any prefix $p$ of $\mathbf u$ of length greater than or equal to $H$ has $lps(p)$ unioccurrent. Assume there are more occurrences of $lps(p)$
in $p$ and consider its suffix $v$ starting in the last-but-one occurrence of $lps(p)$. Since the length of $lps(p)$ is greater than or equal to $N$, the factor $v$ is a~palindrome by assumptions, which contradicts the choice of $lps(p)$.
 \end{proof}

\begin{proof}[Proof of Theorem~\ref{finiteness}]
For periodic words, the statement was shown in~\cite{BrRe-conjecture}.
If $\uu$ is aperiodic, then the statement is a~direct consequence of Lemma~\ref{lps_crw}, Corollary~\ref{charakterizace_defekt}, and Proposition~\ref{charakterizace_T}.
\end{proof}

\section*{Acknowledgments}

This work was supported by the Czech Science Foundation grant GA\v
CR 201/09/0584 and  by the grant MSM6840770039 of the Ministry of
Education, Youth, and Sports of the Czech Republic.

We would like to thank Bojan Ba\v si\'c for his careful reading, useful remarks and in particular and correction of a~lemma.

\end{document}